\documentclass[preprint,12pt,nonatbib]{elsarticle}
\usepackage{rotating}
\usepackage{amsthm,amsmath,amssymb}
\usepackage{mathdots}
\usepackage{graphicx}
\usepackage{mathrsfs}
\usepackage{amsfonts}
\usepackage{enumerate}
\usepackage{latexsym}
\usepackage{graphpap}
\usepackage{cite}
\usepackage{tikz}
\usepackage{pgfplots}
\usepackage{subcaption}
\usepackage{bm}
\usepackage{multirow}
\usepackage{makecell}
\usepackage{blkarray}

\usepackage{float}
\usepackage{array}
\usepackage{booktabs}
\usepackage{caption}
\usepackage{threeparttable}
\usepackage{tabularx}
\usepackage{longtable}
\usepackage{multirow}
\usepackage{diagbox}
\usepackage{rotating}

\theoremstyle{plain}
\newtheorem{theorem}{Theorem}[section]
\newtheorem{corollary}[theorem]{Corollary}
\newtheorem{lemma}[theorem]{Lemma}
\newtheorem{proposition}[theorem]{Proposition}

\numberwithin{equation}{section}

\newcommand{\dom}{\rm{dom}\,}
\newcommand{\im}{\rm{im}\,}

\title{Fixed points of orientation-preserving full transformation}

\author{Yang An,
Wen Ting Zhang\footnote{Corresponding author.}, Yi He\\
\small School of Mathematics and Statistics, Lanzhou University, PR China\\
\small{\tt Email: any2024@lzu.edu.cn; zhangwt@lzu.edu.cn;  heyi2023@lzu.edu.cn}}
\date{}
\journal{arXiv}

\begin{document}
\begin{abstract}
Let $\mathcal{OP}_n$ be the monoid of all orientation-preserving full transformations on $X_n=\{1,\dots, n\}$ with the natural order.
For $\alpha \in \mathcal{OP}_n$, let $F(\alpha)=\{y\in X_n: y\alpha=y\}$ and $F(n,m)=|\{\alpha:|F(\alpha)|=m\}|$. Umar posed the question about the number $F(n,m)$ of elements of $\mathcal{OP}_n$ with $m$ fixed points. In this paper, we show that the number $F(n,m)$ of $\mathcal{OP}_n$ is $\binom{2n}{n-m}$ for $2\leqslant m\leqslant n$ and get the expectation and probability distribution of the cardinality of fixed-point set $F(\alpha)$ for $\alpha\in\mathcal{OP}_n$.
\end{abstract}

\begin{keyword} orientation-preserving\sep full transformation \sep fixed points \sep expectation
\vskip0.2cm
\MSC [2020]  20M20\sep 05A15\sep 60C05
\end{keyword}
\maketitle

\section{Introduction}
Let $X_n=\{1,\ldots,n\}$ be equipped with the natural order. Denote by $\mathcal{PT}_n$ the monoid of all partial transformations on $X_n$. In this paper, we denote composition simply by juxtaposition and shall adopt a right mapping convention; thus if $\alpha,\beta\in\mathcal{PT}_n$, the symbol $\alpha\beta$ denotes the mapping in $\mathcal{PT}_n$ obtained by performing first $\alpha$ and then $\beta$. A transformation $\alpha\in\mathcal{PT}_n$ is called \textit{order-preserving} if, for all $x, y\in\dom\alpha$, $x\leqslant y$ implies $x\alpha\leqslant y\alpha$. Let $\mathcal{PO}_n$ be the submonoid of $\mathcal{PT}_n$ of all order-preserving partial transformations. Let $\alpha\in \mathcal{PT}_n$ and $\dom\alpha$$=\{s_1,\ldots,s_m\}$ with $m\geqslant 0$ and $s_1<s_2<\cdots<s_m$. Then $\alpha$ is called \textit{orientation-preserving} if the sequence of its images $(s_1\alpha,\ldots,s_m\alpha)$ is cyclic, that is, there exists no more than one index $i\in\{1,\ldots,m\}$ such that $s_i\alpha>s_{i+1}\alpha$, where $s_{m+1}$ denotes $s_1$. Let $\mathcal{POP}_n$ be the submonoid of $\mathcal{PT}_n$ of all orientation-preserving partial transformations. Clearly $\mathcal{PO}_n\subseteq\mathcal{POP}_n$. Denote by $\mathcal{T}_n$ and $\mathcal{I}_n$ the submonoid of $\mathcal{PT}_n$ of all full and injective partial transformations, respectively. Let $\mathcal{O}_n=\mathcal{PO}_n\cap\mathcal{T}_n$ and $\mathcal{OP}_n=\mathcal{POP}_n\cap\mathcal{T}_n$ be the submonoids of $\mathcal{T}_n$ of all full order-preserving and orientation-preserving transformations, respectively, and $\mathcal{POI}_n=\mathcal{PO}_n\cap\mathcal{I}_n$ and $\mathcal{POPI}_n=\mathcal{POP}_n\cap\mathcal{I}_n$ be the submonoids of $\mathcal{I}_n$ of all injective order-preserving and orientation-preserving partial transformations, respectively.

Combinatorial and probabilistic properties of transformation semigroups have been studied extensively, yielding many interesting results. For each $\alpha\in\mathcal{PT}_n$, let $Y_n(\alpha)=|{\im}\alpha|$ given that $|{\dom}\alpha|=n$ and $F(\alpha)=\{y\in X_n:y\alpha=y\}$ be the fixed-point set of $\alpha$. For each $x\in X_n$, define $S(n,x)=|\{\alpha\in\mathcal{PT}_n:x\alpha=x\}|$ and $F(n,m)=|\{\alpha\in\mathcal{PT}_n:|F(\alpha)|=m\}|$. Higgins explored the expectation and variance of $Y_n$ for the monoid $\mathcal{O}_n$ and proved that $E(Y_n)=\frac{n^2}{2n-1}$, $\sigma^2(Y_n)=\frac{(n-1)n^2}{2(2n-1)^2}$\cite{higgins1993}. Further, he proved that $E(|F(\alpha)|)=\frac{4^{n-1}}{\binom{2n-1}{n-1}}$, $S(n,x)=\binom{2x-2}{x-1}\binom{2n-2x}{n-x}$ and $F(n,m)=\frac{m}{n}\binom{2n}{n-m}$ in $\mathcal{O}_n$\cite{higgins1993}. Howie showed the combinatorial properties of $\mathcal{O}_n$ and determined that $|\mathcal{O}_n|=\binom{2n-1}{n-1}$ \cite{howie1971}. Subsequently, Gomes and Howie proved that $|\mathcal{PO}_n|=\sum_{r=0}^n\binom{n}{r}\binom{n+r-1}{r}$, the rank of $\mathcal{O}_n$ is $n$ and the rank of $\mathcal{PO}_n$ is $2n-1$ for $n\geqslant2$ \cite{howie1992}. Moreover, Catarino and Higgins studied combinatorial properties of $\mathcal{OP}_n$ and proved that $|\mathcal{OP}_n|=n\binom{2n-1}{n-1}-n(n-1)$ \cite{higgins1999}. Laradji and Umar computed the cardinalities of some equivalence classes in $\mathcal{O}_n$ \cite{umar2006}. For the injective case, Garba studied the combinatorial properties of $\mathcal{POI}_n$ and showed that $|\mathcal{POI}_n|=\binom{2n}{n}$ \cite{garba1994}. Laradji and Umar explored the number $F(n,m)$ of elements of $\mathcal{POI}_n$ with $m$ fixed points and showed that the generating function of $F(n,m)$ is $\frac{x^m}{(x+\sqrt{1-4x})^{m+1}}$ \cite{umar2015}. In {\rm\cite[Tables 3.2 and 6.2]{umar2014}}, Umar posed open problems related to the number $F(n,m)$ in $\mathcal{PO}_n$ and $\mathcal{OP}_n$. Recently, Laradji studied the number $F(n,m)$ of elements of $\mathcal{PO}_n$ with $m$ fixed points and proved that the generating function of $F(n,m)$ is $\frac{16x(1+x-\sqrt{1-6x+x^2})^{m-1}}{(1+x+3\sqrt{1-6x+x^2})^{m+1}}$ \cite{laradji2022}. However, the number $F(n,m)$ of elements of $\mathcal{OP}_n$ with $m$ fixed points has been unknown.

In this paper, we study the combinatorial and probabilistic properties of fixed points in $\mathcal{OP}_n$. In Section \ref{Section2}, we study the number $S(n,x)$ of elements of $\mathcal{OP}_n$ with fixed point $x$ and prove that $S(n,x)=\binom{2n-1}{n-1}-(n-1)$. Further, we compute that the expectation of $|F(\alpha)|$ for $\alpha\in\mathcal{OP}_n$ is $1$. In Section \ref{Section3}, we show that the number $F(n,m)$ of elements of $\mathcal{OP}_n$ with $m$ fixed points is $\binom{2n}{n-m}$ for $2\leqslant m\leqslant n$. Moreover, we get the probability distribution of $|F(\alpha)|$ for $\alpha\in\mathcal{OP}_n$ and verify that $E(|F(\alpha)|)=1$.

\section{Preliminaries}
Refer to the monograph by Howie \cite{howie1995} for any undefined notation and terminology of semigroup theory. We list some known combinatorial results that will be needed later.

\begin{proposition} For natural numbers $m, n, i$ and $j$, the following identities hold
\begin{align}
\sum_{j=0}^n\binom{m+j}{m}
&=\binom{m+n+1}{n},\label{1}\\
\sum_{i=1}^n\binom{2n-i-1}{n-i}
&=\binom{2n-1}{n-1},\label{2}\\
\sum_{m=2}^n\binom{2n}{n-m}
&=2^{2n-1}-\frac{1}{2}\binom{2n}{n}-\binom{2n}{n-1},\label{3}\\
\sum_{m=2}^n m \binom{2n}{n-m}
&=\frac{n-1}{2}\binom{2n}{n-1}.\label{4}
\end{align}
\end{proposition}

\begin{proof}
The identity (\ref{1}) is Hockey Stick Identity \cite{riordan1968}. The identities (\ref{2}), (\ref{3}) and (\ref{4}) hold by
\[\sum_{i=1}^n\binom{2n-i-1}{n-i}=\sum_{j=0}^{n-1}\binom{n+j-1}{j}=\sum_{j=0}^{n-1}\binom{n-1+j}{n-1}\stackrel{(\ref{1})}{=}\binom{2n-1}{n-1},\]
\[\sum_{m=2}^n\binom{2n}{n-m}=\sum_{j=0}^{n-2}\binom{2n}{j}=\sum_{j=0}^{n-1}\binom{2n}{j}-\binom{2n}{n-1}=2^{2n-1}-\frac{1}{2}\binom{2n}{n}-\binom{2n}{n-1},\]
\begin{align*}
\sum_{m=2}^n m \binom{2n}{n-m}
&\;=\;\;\sum_{j=0}^{n-2}(n-j)\binom{2n}{j}=n\sum_{j=0}^{n-2}\binom{2n}{j}-\sum_{j=0}^{n-2}j\binom{2n}{j}\\
&\;=\;\;n\sum_{j=0}^{n-2}\binom{2n}{j}-2n\sum_{j=0}^{n-3}\binom{2n-1}{j}\\
&\;=\;\;n(\sum_{j=0}^{n-2}\binom{2n}{j}-(2^{2n-1}-2\binom{2n}{n-1}))\\
&\stackrel{(\ref{3})}{=}n((2^{2n-1}-\frac{1}{2}\binom{2n}{n}-\binom{2n}{n-1})-(2^{2n-1}-2\binom{2n}{n-1}))\\
&\;=\;\;\frac{n-1}{2}\binom{2n}{n-1}.
\end{align*}
\end{proof}

We list some results for $\mathcal{O}_n$ and $\mathcal{OP}_n$ that will be needed later. Let the $n$-cycle $\delta=(1 2\cdots n)$. Since the sequence of images $(2,\ldots,n,1)$ of the ordered set $X_n$ under $\delta$ is cyclic, $\delta$ lies in $\mathcal{OP}_n$.

\begin{theorem}
{\rm\cite[Theorem 3.12]{higgins1993}}\label{On}
The number $F(n,m)$ of elements of $\mathcal{O}_n$ with $m$ fixed points is $\frac{m}{n}\binom{2n}{n-m}$.
\end{theorem}

\begin{theorem}
{\rm\cite[Theorem 2.6]{higgins1999}}\label{factorization}
Each $\alpha \in \mathcal{OP}_n$ has a factorization $\alpha = \delta^{r}\beta$, where $0\leqslant r\leqslant n-1$ and $\beta\in O_n$, which is unique unless $\alpha$ is constant.
\end{theorem}

\begin{proposition}
{\rm\cite[Corollary 2.7]{higgins1999}}\label{|OPn|}
$|\mathcal{OP}_n|=n|\mathcal{O}_n|-n(n-1)=n\binom{2n-1}{n-1}-n(n-1)$.
\end{proposition}

For each $0\leqslant k\leqslant n-1$, define a total order $\leqslant_k$ on $X_n$ by $k+1\leqslant_k k+2\leqslant_k \cdots\leqslant_k n\leqslant_k 1\leqslant_k\cdots\leqslant_k k$, especially $\leqslant_0$ is just $\leqslant$ as $k=0$. Let $\mathcal{O}_n^k=\{\delta^{-k}\beta\delta^k:\beta\in\mathcal{O}_n\}$. Then $\mathcal{O}_n^k$ is the monoid of all order-preserving transformations on $X_n$ with respect to $\leqslant_k$ by {\rm\cite[Proposition 4.2]{higgins1999}}. Consider the mapping $\varphi_k$ from $\mathcal{O}_n$ to $\mathcal{OP}_n$ by $\beta\varphi_k=\delta^{-k}\beta\delta^k$.

\begin{proposition}
{\rm\cite[Lemma 4.1]{higgins1999}}\label{varphi_k}
The monoid $\mathcal{O}_n^k$ is a submonoid of $\mathcal{OP}_n$ and $\varphi_k$ is an isomorphism from $\mathcal{O}_n$ to $\mathcal{O}_n^k$.
\end{proposition}

\begin{proposition}
{\rm\cite[Lemma 4.7]{higgins1999}}\label{F(h)}
Let $\gamma=\delta^{-k}\beta\delta^k\in\mathcal{O}_n^k$, where $0\leqslant k\leqslant n-1$ and $\beta\in\mathcal{O}_n$. Then $x\in F(\gamma)$ if and only if $x-k\in F(\beta)$ for each $x\in X_n$.
\end{proposition}

It has been observed that each $\alpha\in\mathcal{T}_n$ can be pictured as a digraph on $n$ vertices with $ij$ an arc of $\alpha$ if $i\alpha=j$. Each connected component of such a digraph is functional, meaning that it consists of a unique cycle, together with a number of trees rooted around the points of the cycle. For example, Figure \ref{figure_1} is the digraph of $\alpha=\binom{1\quad2\quad3\quad4\quad5}{2\quad5\quad2\quad5\quad5}$. For further background see \cite{higgins1992}.

\begin{figure}[H]
  \centering
    \begin{tikzpicture}[scale=0.5]
      \node[fill, circle, inner sep=1pt, label={above:1}] (p1) at (0,0) {};

      \node[fill, circle, inner sep=1pt, label={above:2}] (p2) at (2,0) {};

      \node[fill, circle, inner sep=1pt, label={above:3}] (p3) at (4,0) {};

      \node[fill, circle, inner sep=1pt, label={above:4}] (p4) at (6,0) {};

      \node[fill, circle, inner sep=1pt, label={above:5}] (p5) at (8,0) {};
      \draw (8.5,0) circle (0.5);
      \draw[thick,->] (9.0,0) arc (0:690:0.5);

      \draw[thick,->] (p1) -- (p2);
      \draw[thick,->] (p3) -- (p2);
      \draw[thick,->] (p4) -- (p5);
      \draw[thick,->] (p2) to[out=-45, in=-135] (p5);
    \end{tikzpicture}
    \caption{Digraph of $\alpha=\binom{1\quad2\quad3\quad4\quad5}{2\quad5\quad2\quad5\quad5}$}
  \label{figure_1}
\end{figure}

\begin{proposition}
{\rm\cite[Lemma 4.10]{higgins1999}}\label{interval}
Let $\alpha\in \mathcal{OP}_n$. If $F(\alpha)\neq \varnothing$, then the digraph of $\alpha$ is a forest and each component associated with a fixed point of $\alpha$ is an interval.
\end{proposition}

\begin{proposition}
{\rm\cite[Lemma 4.11 and Corollary 4.12]{higgins1999}}\label{leqslant_i-1}
Let $\alpha\in\mathcal{OP}_n$ such that $F(\alpha)\neq \varnothing$, and let $C$ be any component of $\alpha$. Then $C=[i,j]$ for some $i,j\in X_n$ and $\alpha$ is order-preserving with respect to $\leqslant_{i-1}$.
\end{proposition}

\begin{theorem}
{\rm\cite[Theorem 4.13]{higgins1999}}\label{equal}
$\bigcup_{k=0}^{n-1}\mathcal{O}_n^k=\{\alpha\in\mathcal{OP}_n:F(\alpha)\neq \varnothing\}$.
\end{theorem}

\section{$S(n,x)$ and the expectation in $\mathcal{OP}_n$}\label{Section2}
In this section, we prove that $S(n,x)=\binom{2n-1}{n-1}-(n-1)$ in $\mathcal{OP}_n$. Further, we compute the expectation of $|F(\alpha)|$ for $\alpha\in\mathcal{OP}_n$ and show that $E(|F(\alpha)|)=1$. For all $i,j\in X_n$, let $N(i,j)=|\{\beta\in\mathcal{O}_n:i\beta=j\}|$.

\begin{lemma}
For natural number $1\leqslant i,j\leqslant n$,
\begin{equation}
\label{N(i,j)}
N(i,j)=\binom{i+j-2}{i-1}\binom{2n-i-j}{n-i}.
\end{equation}
\end{lemma}

\begin{proof}
The number of transformations $\beta\in\mathcal{O}_n$ with $i\beta=j$ for all $i,j \in X_n$ is the product of the
number of order-preserving transformations from $X_{i-1}$ to $X_j$ and the number of such transformations from $X_{n-i}$ to $X_{n-j+1}$. Note that the number of order-preserving transformations from a set of order $r$ to a set of order $m$ is $\binom{m+r-1}{r}$. Therefore the number $N(i,j)$ of $\beta\in\mathcal{O}_n$ with $i\beta=j$ is $\binom{i+j-2}{i-1}\binom{2n-i-j}{n-i}$.\qedhere
\end{proof}

\begin{lemma}\label{S(n,1)}
The number of elements of $\mathcal{OP}_n$ with fixed point $1$ is
\[S(n,1)=\binom{2n-1}{n-1}-(n-1).\]
\end{lemma}

\begin{proof}
For each $x\in X_n$, consider the mapping
\[\varphi:\{(r,\beta):0\leqslant r\leqslant n-1,\beta\in O_n,x\delta^r\beta=x\} \longrightarrow \{\alpha\in\mathcal{OP}_n: x\alpha=x\}\]
by $(r,\beta)\varphi=\alpha$. By Proposition \ref{factorization}, $\varphi$ is a bijection when $\alpha$ is non-constant. If $\alpha$ is constant with fixed point $x$, then $r=0,1,\ldots,n-1$ and $\beta$ is constant with $x$. Let $x=1$. The number of elements of $\mathcal{OP}_n$ with fixed point $1$ is
\begin{align*}
S(n,1)
&\;=\;\;\sum_{r=0}^{n-1}|\{\beta\in O_n: (1+r)\beta=1\}|-(n-1) =\sum_{i=1}^n|N(i,1)|-(n-1)\\
&\stackrel{(\ref{N(i,j)})}{=}\sum_{i=1}^n\binom{2n-i-1}{n-i}-(n-1)\stackrel{(\ref{2})}{=}\binom{2n-1}{n-1}-(n-1).\qedhere
\end{align*}
\end{proof}

\begin{theorem}\label{S(n,x)}
The number of elements of $\mathcal{OP}_n$ with fixed point $x$ is
\[S(n,x)=\binom{2n-1}{n-1}-(n-1)\]
where $1\leqslant x\leqslant n$.
\end{theorem}

\begin{proof}
For $0\leqslant r\leqslant n-1$, since the sequences of images $(1+r,2+r,\ldots, n+r)$ and $(1-r,2-r,\ldots, n-r)$ of the ordered set $X_n$ under $\delta^r$ and $\delta^{-r}$ are cyclic respectively, $\delta^r$ and $\delta^{-r}$ lie in $\mathcal{OP}_n$. Thus $\delta^{-r}\alpha\delta^r\in \mathcal{OP}_n$ for $\alpha\in\mathcal{OP}_n$. Consider the mappings $\varphi_r$ and $\varphi_{-r}$ from $\mathcal{OP}_n$ to $\mathcal{OP}_n$ by $\alpha\varphi_r=\delta^{-r}\alpha \delta^r$ and $\alpha\varphi_{-r}=\delta^r\alpha \delta^{-r}$,
\[\alpha\varphi_r\varphi_{-r}=\delta^r(\delta^{-r}\alpha\delta^r)\delta^{-r}=\alpha \quad\text{and}\quad \alpha\varphi_{-r}\varphi_r=\delta^{-r}(\delta^r\alpha\delta^{-r})\delta^r=\alpha\]
for each $\alpha\in\mathcal{OP}_n$. Hence $\varphi_r$ has a two-sided inverse $\varphi_{-r}$ and is a bijection.

For each $\alpha\in\mathcal{OP}_n$, if $1\alpha=1$, then
\[(1+r)\alpha\varphi_r=(1+r)\delta^{-r}\alpha\delta^r=1\alpha \delta^r=1\delta^r=1+r\]
where $0\leqslant r\leqslant n-1$. Hence $\alpha\varphi_r$ is also fixed at the point $x=1+r$, where $1\leqslant x\leqslant n$.
Since $\varphi_r$ is a bijection, $S(n,x)=S(n,1)=\binom{2n-1}{n-1}-(n-1)$.\qedhere
\end{proof}

Some computed values of $S(n,x)$ for $1\leqslant x \leqslant n \leqslant 6$ are shown in Table \ref{table S(n,x)}.

\begin{table}[H]
\centering
\begin{threeparttable}
\setlength{\tabcolsep}{8pt}
\renewcommand{\arraystretch}{1.20}
\begin{tabular}{lccccccc}
\toprule
$n\backslash x$ & 1 & 2 & 3 & 4 & 5 & 6 & $\sum_xS(n,x)=|\mathcal{OP}_n|$ \\
\midrule
1 & 1   &     &     &     &     &     & 1 \\
2 & 2   & 2   &     &     &     &     & 4 \\
3 & 8   & 8   & 8   &     &     &     & 24 \\
4 & 32  & 32  & 32  & 32  &     &     & 128 \\
5 & 122 & 122 & 122 & 122 & 122 &     & 610 \\
6 & 457 & 457 & 457 & 457 & 457 & 457 & 2742 \\
\bottomrule
\end{tabular}
\caption{$S(n,x)$ for $1\leqslant x \leqslant n \leqslant 6$}
\label{table S(n,x)}
\end{threeparttable}
\end{table}

\begin{theorem}
For each $\alpha\in\mathcal{OP}_n$, the probability function of $x\alpha=x$ and the expectation of $|F(\alpha)|$ are
\[P(x\alpha=x)=\frac{1}{n} \quad\text{and}\quad E(|F(\alpha)|)=1\]
where $1\leqslant x\leqslant n$.
\end{theorem}

\begin{proof}
By Proposition \ref{|OPn|} and Theorem \ref{S(n,x)}, the probability function of $x\alpha=x$ is
\[P(x\alpha=x)=\frac{S(n,x)}{|\mathcal{OP}_n|}=\frac{\binom{2n-1}{n-1}-(n-1)}{n\binom{2n-1}{n-1}-n(n-1)}=\frac{1}{n}.\]
Hence the expectation of $|F(\alpha)|$ is
\[E(|F(\alpha)|)=\sum_{x=1}^nP(x\alpha=x)=n\cdot\frac{1}{n}=1.\qedhere\]
\end{proof}

\section{$F(n,m)$ and the distribution in $\mathcal{OP}_n$}\label{Section3}
In this section, we determine that $F(n,m)=\binom{2n}{n-m}$ with $2\leqslant m\leqslant n$ in $\mathcal{OP}_n$. Moreover, we get the probability distribution of $|F(\alpha)|$ for $\alpha\in\mathcal{OP}_n$ and verify $E(|F(\alpha)|)=1$. For each $\alpha\in\mathcal{OP}_n$, let $K(\alpha)=\{k:0\leqslant k\leqslant n-1, \alpha\in\mathcal{O}_n^k\}$.

\begin{lemma}\label{K}
For each $\alpha\in\mathcal{OP}_n$ with $m\geqslant2$ fixed points, the number of $\mathcal{O}_n^k$ containing $\alpha$ is
\[|K(\alpha)|=m.\]
Further, let $C_r=[i_r,j_r]$ be the component associated with one of fixed points of $\alpha$. Then $k=i_r-1$ for $1\leqslant r\leqslant m$.
\end{lemma}

\begin{proof}
For each $\alpha\in\mathcal{OP}_n$ with $m\geqslant2$ fixed points, $\alpha$ is order-preserving with respect to $\leqslant_{i_r-1}$ by Proposition \ref{leqslant_i-1}. Recall that $\mathcal{O}_n^{i_r-1}$ is the monoid of all order-preserving transformations on $X_n$ with respect to $\leqslant_{i_r-1}$. Hence $\alpha\in\mathcal{O}_n^{i_r-1}$ for $1\leqslant r\leqslant m$.  Since different components are disjoint, there are $m$ pairwise distinct values of $k=i_r-1$ and $|K(\alpha)|\geqslant m$.

Let $\alpha\in\mathcal{O}_n^k$ and $\beta=\delta^k\alpha\delta^{-k}\in\mathcal{O}_n$ for $0\leqslant k\leqslant n-1$. Since $\mathcal{O}_n$ is the submonoid of $\mathcal{OP}_n$ by Proposition \ref{varphi_k}, each component associated with a fixed point of $\beta$ is an interval by Proposition \ref{interval}. In particular, the component $D$ of $\beta$ containing $1$ must be $[1,t]$ for $1\leqslant t\leqslant n$, where $1$ is the minimum element with respect to $\leqslant$. Moreover, the component $D$ of $\beta$ is corresponding to some component $C_r=[i_r,j_r]$ of $\alpha$, where $i_r=k+1$. That is, each $k$ for which $\alpha\in\mathcal{O}_n^k$ arises from the left endpoint $i_r$ of some component, and is determined by the relation $k=i_r-1$. Hence $|K(\alpha)|\leqslant m$. Therefore, $|K(\alpha)|=m$ and $k=i_r-1$ for $1\leqslant r\leqslant m$.\qedhere
\end{proof}

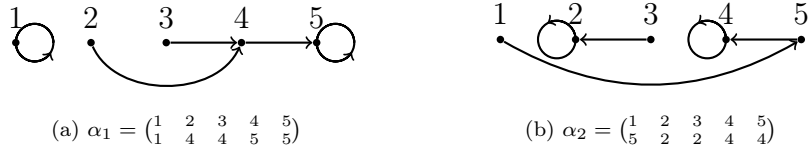
\begin{figure}[H]
\centering
  \begin{subfigure}{0.45\textwidth}
  \centering
    \begin{tikzpicture}[scale=0.5]
      \node[fill, circle, inner sep=1pt, label={above:1}] (p1_1) at (0,0) {};
      \draw (0.5,0) circle (0.5);
      \draw[thick,->] (1.0,0) arc (0:690:0.5);

      \node[fill, circle, inner sep=1pt, label={above:2}] (p2_1) at (2,0) {};

      \node[fill, circle, inner sep=1pt, label={above:3}] (p3_1) at (4,0) {};

      \node[fill, circle, inner sep=1pt, label={above:4}] (p4_1) at (6,0) {};

      \node[fill, circle, inner sep=1pt, label={above:5}] (p5_1) at (8,0) {};
      \draw (8.5,0) circle (0.5);
      \draw[thick,->] (9.0,0) arc (0:690:0.5);

      \draw[thick,->] (p3_1) -- (p4_1);
      \draw[thick,->] (p4_1) -- (p5_1);
      \draw[thick,->] (p2_1) to[out=-66, in=-114] (p4_1);
    \end{tikzpicture}
   \caption{$\alpha_1=\binom{1\quad2\quad3\quad4\quad5}{1\quad4\quad4\quad5\quad5}$}
  \end{subfigure}
  \begin{subfigure}{0.45\textwidth}
  \centering
    \begin{tikzpicture}[scale=0.5]
      \node[fill, circle, inner sep=1pt, label={above:1}] (p1_1) at (0,0) {};

      \node[fill, circle, inner sep=1pt, label={above:2}] (p2_1) at (2,0) {};
      \draw (1.5,0) circle (0.5);
      \draw[thick,->] (2.0,0) arc (0:480:0.5);

      \node[fill, circle, inner sep=1pt, label={above:3}] (p3_1) at (4,0) {};

      \node[fill, circle, inner sep=1pt, label={above:4}] (p4_1) at (6,0) {};
      \draw (5.5,0) circle (0.5);
      \draw[thick,->] (6.0,0) arc (0:480:0.5);

      \node[fill, circle, inner sep=1pt, label={above:5}] (p5_1) at (8,0) {};

      \draw[thick,->] (p3_1) -- (p2_1);
      \draw[thick,->] (p5_1) -- (p4_1);
      \draw[thick,->] (p1_1) to[out=-30, in=-150] (p5_1);
    \end{tikzpicture}
  \caption{$\alpha_2=\binom{1\quad2\quad3\quad4\quad5}{5\quad2\quad2\quad4\quad4}$}
  \end{subfigure}
\caption{Digraphs of $\alpha\in\mathcal{OP}_n$ with $2$ fixed points}
\label{figure_2}
\end{figure}

For example, the left side in Figure \ref{figure_2} is the digraph of $\alpha_1=\binom{1\quad2\quad3\quad4\quad5}{1\quad4\quad4\quad5\quad5}$ with $F(\alpha_1)=\{1,5\}$, which has two components $C_1=\{1\}=[1,1]$ and $C_2=\{2,3,4,5\}=[2,5]$. By Lemma \ref{K}, $\alpha_1\in\mathcal{O}_n$ and $\alpha_1\in\mathcal{O}_n^1$ and so $|K(\alpha_1)|=2$. This is consistent with the fact that $\alpha_1$ is order-preserving with respect to $\leqslant$ and $\leqslant_1$, $\alpha_1(1)\leqslant\alpha_1(2)\leqslant\alpha_1(3)\leqslant\alpha_1(4)\leqslant\alpha_1(5)$ and $\alpha_1(2)\leqslant_1\alpha_1(3)\leqslant_1\alpha_1(4)\leqslant_1\alpha_1(5)\leqslant_1\alpha_1(1)$. The right side in Figure \ref{figure_2} is the digraph of $\alpha_2=\binom{1\quad2\quad3\quad4\quad5}{5\quad2\quad2\quad4\quad4}$ with $F(\alpha_2)=\{2,4\}$, which has two components $C_1=\{1,4,5\}=[4,1]$ and $C_2=\{2,3\}=[2,3]$. By Lemma \ref{K}, $\alpha_2\in\mathcal{O}_n^3$ and $\alpha_2\in\mathcal{O}_n^1$ and so $|K(\alpha_2)|=2$. This is consistent with the fact that $\alpha_2$ is order-preserving with respect to $\leqslant_3$ and $\leqslant_1$, $\alpha_2(4)\leqslant_3\alpha_2(5)\leqslant_3\alpha_2(1)\leqslant_3\alpha_2(2)\leqslant_3\alpha_2(3)$ and $\alpha_2(2)\leqslant_1\alpha_2(3)\leqslant_1\alpha_2(4)\leqslant_1\alpha_2(5)\leqslant_1\alpha_2(1)$.

\begin{lemma}\label{F(n,m)_O_n^k}
The number $F(n,m)$ of elements of $\mathcal{O}_n^k$ with $m$ fixed points is $\frac{m}{n}\binom{2n}{n-m}$.
\end{lemma}

\begin{proof}
By Proposition \ref{F(h)}, for $\gamma\in\mathcal{O}_n^k$ with $0\leqslant k\leqslant n-1$ and $\beta\in\mathcal{O}_n$, since $x\in F(\gamma)$ if and only if $x-k\in F(\beta)$, the number of $\gamma$ with $m$ fixed points is equal to the number of $\beta$ with $m$ fixed points. By Theorem \ref{On}, the number $F(n,m)$ of elements of $\mathcal{O}_n$ with $m$ fixed points is $\frac{m}{n}\binom{2n}{n-m}$. Hence the number $F(n,m)$ of elements of $\mathcal{O}_n^k$ with $m$ fixed points is $\frac{m}{n}\binom{2n}{n-m}$.\qedhere
\end{proof}

\begin{theorem}\label{F(n,m)}
The number of elements of $\mathcal{OP}_n$ with $m$ fixed points is
\[F(n,m)=\binom{2n}{n-m}\]
where $2\leqslant m\leqslant n$.
\end{theorem}

\begin{proof}
Since $\bigcup_{k=0}^{n-1}\mathcal{O}_n^k=\{\alpha\in\mathcal{OP}_n:F(\alpha)\neq \varnothing\}$ by Theorem \ref{equal}, the set of $\alpha\in\mathcal{OP}_n$ with $m\geqslant 2$ fixed points is
\[\{\alpha\in\mathcal{OP}_n:|F(\alpha)|=m\}=\bigcup_{k=0}^{n-1}\{\gamma\in\mathcal{O}_n^k:|F(\gamma)|=m\}.\]
By Lemmas \ref{K} and \ref{F(n,m)_O_n^k}, the number of elements of $\mathcal{OP}_n$ with $m\geqslant 2$ fixed points is
\[F(n,m)=\frac{\sum_{k=0}^{n-1}|\{\gamma\in\mathcal{O}_n^k:|F(\gamma)|=m\}|}{|K(\alpha)|}=\frac{n}{m}\cdot\frac{m}{n}\binom{2n}{n-m}=\binom{2n}{n-m}.\qedhere\]
\end{proof}

\begin{theorem}\label{F(n,0) and F(n,1)}
The number of elements of $\mathcal{OP}_n$ with one fixed point and without any fixed point are
\[F(n,1)=\binom{2n}{n-1}-n(n-1),\]
\[F(n,0)=(n+1)\binom{2n-1}{n-1}-2^{2n-1}.\]
\end{theorem}

\begin{proof}
Assume that $\alpha\in\mathcal{OP}_n$ has only one fixed point $x$, where $1\leqslant x\leqslant n$. Consider the mapping
\[\varphi:\{(r,\beta):0\leqslant r\leqslant n-1,\beta\in O_n,x\delta^r\beta=x\} \longrightarrow \{\alpha\in\mathcal{OP}_n: x\alpha=x\}\]
by $(r,\beta)\varphi=\alpha$. By Proposition \ref{factorization}, $\varphi$ is a bijection when $\alpha$ is non-constant. If $\alpha$ is constant, then $r=0,1,\ldots,n-1$ and $\beta$ is constant with $x$. Since constant transformation must be the transformation with one fixed point, the number of elements of $\mathcal{OP}_n$ with one fixed point is
\[F(n,1)=\binom{2n}{n-1}-n(n-1).\]
By Proposition \ref{|OPn|}, the number of elements of $\mathcal{OP}_n$ without any fixed point is
\begin{align*}
F(n,0)
&\;=\;\;|\mathcal{OP}_n|-F(n,1)-\sum_{m=2}^nF(n,m)\\
&\;=\;\;n\binom{2n-1}{n-1}-n(n-1)-(\binom{2n}{n-1}-n(n-1))-\sum_{m=2}^n\binom{2n}{n-m}\\
&\stackrel{(\ref{3})}{=}n\binom{2n-1}{n-1}-\binom{2n}{n-1}-(2^{2n-1}-\frac{1}{2}\binom{2n}{n}-\binom{2n}{n-1})\\
&\;=\;\;(n+1)\binom{2n-1}{n-1}-2^{2n-1}.\qedhere
\end{align*}
\end{proof}

Some computed values of $F(n,m)$ for $0\leqslant m \leqslant n \leqslant 6$ are shown in Table \ref{table F(n,m)}.

\begin{table}[H]
\centering
\begin{threeparttable}
\setlength{\tabcolsep}{8pt}
\renewcommand{\arraystretch}{1.20}
\begin{tabular}{lcccccccc}
\toprule
$n\backslash m$ & 0 & 1 & 2 & 3 & 4 & 5 & 6 & $\sum_mF(n,m)=|\mathcal{OP}_n|$ \\
\midrule
1 &      & 1   &     &     &    &    &   & 1 \\
2 & 1    & 2   & 1   &     &    &    &   & 4 \\
3 & 8    & 9   & 6   & 1   &    &    &   & 24 \\
4 & 47   & 44  & 28  & 8   & 1  &    &   & 128 \\
5 & 244  & 190 & 120 & 45  & 10 & 1  &   & 610 \\
6 & 1186 & 762 & 495 & 220 & 66 & 12 & 1 & 2742 \\
\bottomrule
\end{tabular}
\caption{$F(n,m)$ for $0\leqslant m \leqslant n \leqslant 6$}
\label{table F(n,m)}
\end{threeparttable}
\end{table}

\begin{corollary}\label{recurrence relation}
The number of elements of $\mathcal{OP}_n$ with $m$ fixed points has the recurrence relation
\[F(n,m+2)=F(n+1,m+1)-2F(n,m+1)-F(n,m)\]
where $2\leqslant m\leqslant n$.
\end{corollary}

\begin{proof}
This is an immediate consequence of the fact that $F(n,m)=\binom{2n}{n-m}$.\qedhere
\end{proof}

\begin{corollary}
For each $\alpha\in\mathcal{OP}_n$, the probability function of $|F(\alpha)|$ is
\[P(|F(\alpha)|=m)=
\begin{cases}
\frac{(n+1)\binom{2n-1}{n-1}-2^{2n-1}}{n\binom{2n-1}{n-1}-n(n-1)},&\text{if}\;\; m = 0,\\
\frac{\binom{2n}{n-1}-n(n-1)}{n\binom{2n-1}{n-1}-n(n-1)},&\text{if}\;\; m = 1,\\
\frac{\binom{2n}{n-m}}{n\binom{2n-1}{n-1}-n(n-1)},&\text{if}\;\; 2 \leqslant m \leqslant n.
\end{cases}\]
\end{corollary}

\begin{proof}
By Proposition \ref{|OPn|}, Theorems \ref{F(n,m)} and  \ref{F(n,0) and F(n,1)}, the probability function of $|F(\alpha)|$ is
\[P(|F(\alpha)|=m)=\frac{F(n,m)}{|\mathcal{OP}_n|}=
\begin{cases}
\frac{(n+1)\binom{2n-1}{n-1}-2^{2n-1}}{n\binom{2n-1}{n-1}-n(n-1)},&\text{if}\;\; m = 0,\\
\frac{\binom{2n}{n-1}-n(n-1)}{n\binom{2n-1}{n-1}-n(n-1)},&\text{if}\;\; m = 1,\\
\frac{\binom{2n}{n-m}}{n\binom{2n-1}{n-1}-n(n-1)},&\text{if}\;\; 2 \leqslant m \leqslant n.
\end{cases}\]
Hence the expectation of $|F(\alpha)|$ is
\begin{align*}
E(|F(\alpha)|)
&\;=\;\;\sum_{m=0}^n m\cdot P(|F(\alpha)|=m)\\
&\;=\;\;\frac{\binom{2n}{n-1}-n(n-1)}{n\binom{2n-1}{n-1}-n(n-1)}+\sum_{m=2}^n\frac{m\binom{2n}{n-m}}{n\binom{2n-1}{n-1}-n(n-1)}\\
&\stackrel{(\ref{4})}{=}\frac{\binom{2n}{n-1}-n(n-1)+\frac{n-1}{2}\binom{2n}{n-1}}{n\binom{2n-1}{n-1}-n(n-1)}\\
&\;=\;\;\frac{\frac{n+1}{2}\binom{2n}{n-1}-n(n-1)}{n\binom{2n-1}{n-1}-n(n-1)}=1.\qedhere
\end{align*}
\end{proof}

\section*{Acknowledgements}
This research is partially supported by the National Natural Science Foundation of China (Nos. 12271224, 12571018, 12401027), the Fundamental Research Funds for the Central University (No. lzujbky-2023-ey06) and Gansu Provincial Department of Education: Innovation Star Project for Graduate Students in Universities in Gansu Province (No. 2026CXZX-151).

\end{document}